\documentclass[a4paper,11pt]{article}

\usepackage{amsmath,amsthm,amssymb,mathrsfs,latexsym,amsfonts}
\usepackage{graphicx,psfrag,epsfig}
\usepackage[english]{babel}
\usepackage[latin1]{inputenc}
\usepackage{pstricks}

\newtheorem{theorem}{Theorem}[section]
\newtheorem{lemma}[theorem]{Lemma}
\newtheorem{proposition}[theorem]{Proposition}

\newtheorem{corollary}[theorem]{Corollary}
\newtheorem{definition}[theorem]{Definition\rm}

\newcounter{paraga}[section]
\renewcommand{\theparaga}{{\bf\arabic{paraga}.}}
\newcommand{\paraga}{\medskip \addtocounter{paraga}{1} 
\noindent{\theparaga\ } }

\begin{document}

\def\MP{\,{<\hspace{-.5em}\cdot}\,}
\def\SP{\,{>\hspace{-.3em}\cdot}\,}
\def\PM{\,{\cdot\hspace{-.3em}<}\,}
\def\PS{\,{\cdot\hspace{-.3em}>}\,}
\def\EP{\,{=\hspace{-.2em}\cdot}\,}
\def\PP{\,{+\hspace{-.1em}\cdot}\,}
\def\PE{\,{\cdot\hspace{-.2em}=}\,}
\def\N{\mathbb N}
\def\C{\mathbb C}
\def\Q{\mathbb Q}
\def\R{\mathbb R}
\def\T{\mathbb T}
\def\A{\mathbb A}
\def\Z{\mathbb Z}
\def\demi{\frac{1}{2}}

\begin{titlepage}
\author{Abed Bounemoura~\footnote{A.Bounemoura@warwick.ac.uk, Mathematics Institute, University of Warwick}}
\title{\LARGE{\textbf{Effective stability for Gevrey and finitely differentiable prevalent Hamiltonians}}}
\end{titlepage}

\maketitle

\begin{abstract}
For perturbations of integrable Hamiltonians systems, the Nekhoroshev theorem shows that all solutions are stable for an exponentially long interval of time, provided the integrable part satisfies a steepness condition and the system is analytic. This fundamental result has been extended in two distinct directions. The first one is due to Niederman, who showed that under the analyticity assumption, the result holds true for a prevalent class of integrable systems which is much wider than the steep systems. The second one is due to Marco-Sauzin but it is limited to quasi-convex integrable systems, for which they showed exponential stability if the system is assumed to be only Gevrey regular. If the system is finitely differentiable, the author showed polynomial stability, still in the quasi-convex case. The goal of this work is to generalize all these results in a unified way, by proving exponential or polynomial stability for Gevrey or finitely differentiable perturbations of prevalent integrable Hamiltonian systems.   
\end{abstract}
 
\section{Introduction}

\paraga Consider a near-integrable Hamiltonian system, that is a perturbation of an integrable Hamiltonian system, which is of the form
\begin{equation*}
\begin{cases} 
H(\theta,I)=h(I)+f(\theta,I), \\
|h|=1, \; |f| < \varepsilon <\!\!<1. 
\end{cases}
\end{equation*}
Here $(\theta,I) \in \T^n \times \R^n$ are angle-action coordinates, and $f$ is a small perturbation, of size $\varepsilon$, in some suitable topology defined by a norm $|\,.\,|$. In the absence of perturbation, that is when $\varepsilon$ is zero, the action variables $I(t)$ are integrals of motions and all solutions are quasi-periodic. Therefore it is a natural question, which is in fact motivated by concrete problems of stability in celestial mechanics, to study the evolution of the action variables $I(t)$ after perturbation, that is for $\varepsilon>0$ but arbitrarily small.

\paraga If the system is analytic and if $h$ satisfies a steepness condition, which is a quantitative transversality condition, it is a remarkable result due to Nekhoroshev (\cite{Nek77}, \cite{Nek79}) that the action variables are stable for an exponentially long interval of time with respect to the inverse of the size of the perturbation: one has
\[ |I(t)-I_0| \leq c_1\varepsilon^b, \quad |t|\leq \exp(c_2\varepsilon^{-a}), \] 
for some positive constants $c_1,c_2,a,b$ and provided that the size of the perturbation $\varepsilon$ is smaller than a threshold $\varepsilon_0$. This is a major result in the theory of perturbations of Hamiltonian systems, and it gives complement to two previous and equally important results. The first one is that ``most" solutions (in a measure-theoretical sense) are quasi-periodic and hence stable for all time, in the sense that
\[ |I(t)-I_0|\leq c\sqrt{\varepsilon}, \quad t\in\R, \]
for some positive constant $c$, provided that the system is sufficiently regular (analytic, Gevrey or even $C^k$ for $k>2n$), $h$ satisfies a mild non-degeneracy assumption and $\varepsilon$ is smaller than a threshold $\varepsilon_0$. This is the content of KAM theory, see \cite{Kol54} for the original statement and, for instance \cite{Rus01}, \cite{Sal04} and \cite{Pop04} among the enormous literature, for various improvements regarding the hypotheses of non-degeneracy and regularity. Hence Nekhoroshev estimates give new information for solutions living on the phase space not covered by KAM theory, and even though the latter has a small measure, it is usually topologically large (at least for $n\geq 3$). The second important previous result is that there exist ``unstable" solutions, satisfying
\[ |I(\tau)-I_0|\geq 1, \quad \tau=\tau(\varepsilon)>0. \]
This was discovered by Arnold in his famous paper \cite{Arn64}, and it has become widely known as ``Arnold diffusion". Even though this phenomenon has been intensively studied, very little is known. Here Nekhoroshev estimates give an exponentially large lower bound on the time of instability $\tau(\varepsilon)$, explaining (in part) why such instability properties are so hard to detect.  

\paraga Now returning to Nekhoroshev estimates, the original proof is rather long and complicated. It is naturally divided into two parts. The first part, which is analytic, is the construction of general resonant normal forms, up to an exponentially small remainder (this is where the analyticity of the system is used), on local domains of the action space where one has a suitable control on the so-called small divisors. Then, the second part, which is geometric, consists in the construction of a partition of the action space where one can use such normal forms. This is where the steepness of the integrable system enters, basically it rules out the existence of solutions which cannot be controlled by these normal forms, so eventually the exponentially small remainders easily translate into an exponentially long time of stability for all solutions.   

\paraga It has been noticed by the Italian school (\cite{BGG85}, \cite{BG86}) that using preservation of energy, the geometric part of the proof can be simplified for the simplest steep Hamiltonians, namely strictly convex or strictly quasi-convex Hamiltonians (recall that quasi-convexity means that the energy sub-levels are convex subsets). Then, much work have been devoted to this special case. In particular, Lochak introduced in \cite{Loc92} a new method leading in particular to an extremely simple and elegant proof of these estimates under the quasi-convexity assumption. His approach only relies on averaging along periodic frequencies, which enables him to construct special resonant normal forms (periodic frequencies are just resonant frequencies of codimension-one multiplicities), and the use of the most basic result in simultaneous Diophantine approximation, namely Dirichlet's theorem, to cover the whole action space by domains where such special normal forms can be used. This also brought to light the surprising phenomenon of ``stabilization by resonances", which implies that the more resonant the initial condition is, the more stable (in a finite time-scale) the solution will be. This method had several applications, and an important one, that we shall be concerned here, was the extension of these stability estimates for non-analytic systems. Indeed, using Lochak's strategy, it was shown in \cite{MS02} that the exponential estimates are also satisfied for Gevrey regular systems, and in \cite{Bou10} it was proved that polynomial estimates hold true if the system is only of finite differentiability (this is obviously the best one can expect under such a weak regularity assumption). Let us point out that for non-analytic systems, the construction of these normal forms, which is the only new ingredient one has to add since the geometric part of the proof is insensible to the regularity of the system, is more difficult than for an analytic system (this is especially true for Gevrey systems). Indeed, one cannot simply work with $C^0$-norms (on some complex strip, then by the usual Cauchy estimates one has a control on all derivatives on smaller complex strips), so one has to work directly with all the derivatives (and moreover keep a control on the growth of these derivatives in the Gevrey case). In \cite{MS02} and \cite{Bou10}, such normal forms were constructed but only because it was enough to consider periodic frequencies, which boils down to what is classically known as a one-phase averaging.

\paraga However, all these results were restricted by the quasi-convexity hypothesis. A study of Nekhoroshev estimates under more general assumptions on the integrable part has been initiated by Niederman. First, in \cite{Nie04}, he introduced new geometric arguments based on simultaneous Diophantine approximation, leading to a great simplification in the geometric part of the proof of Nekhoroshev's result under the steepness condition. Then, in \cite{Nie07}, he realized that this method allows in fact to obtain the result for a much wider class of unperturbed Hamiltonians, which he called ``Diophantine steep", and which are prevalent (in the sense of Hunt, Sauer and York, recall that prevalence is a possible generalization of full Lebesgue measure for infinite dimensional linear spaces). However, the analytic part of Niederman's proof was still based on averaging along general frequencies and hence required the construction of general resonant normal forms (which was taken from \cite{Pos93}). In fact, in the non-convex case, a simple averaging along a periodic frequency, which corresponds to studying the dynamics in the neighbourhood of a resonance of codimension-one multiplicity, cannot be enough since solutions will necessarily explore resonances associated to different, and possibly all, multiplicities. But then in \cite{BN09} we were able to construct normal norms associated to any multiplicities by making suitable composition of periodic averagings, with periodic vectors which are independent and sufficiently close to each other. This was an extension of Lochak's method, in the sense that no small divisors were involved, only (a composition of) periodic averagings and simultaneous Diophantine approximation were used. This proof was not only simpler than the previous one, but also opened the way to several applications. For instance, in \cite{Bou09}, it was shown how one can easily obtain more general results of stability in the vicinity of linearly stable quasi-periodic invariant tori.

\paraga The aim of this paper is to extend the above results by proving stability estimates for Hamiltonian systems with a prevalent integrable part, but which are not necessarily analytic. In the Gevrey case, this will lead to exponential estimates of stability for perturbation of a generic integrable Hamiltonian, as stated below.

\begin{theorem} \label{mainth}
For $\alpha\geq 1$, consider an arbitrary $\alpha$-Gevrey integrable Hamiltonian $h$ defined on an open ball in $\R^n$. Then for almost any $\xi\in\R^n$, the integrable Hamiltonian $h_\xi(x)=h(I)-\xi.I$ is exponentially stable.
\end{theorem}

This will be a direct consequence of Theorems~\ref{thmpreva} and Theorem~\ref{thmG} below. This result generalizes the main results of \cite{Nie07} and \cite{BN09} which were restricted by the analyticity assumption ($\alpha=1$), and the main stability result of \cite{MS02} which was restricted by the quasi-convexity assumption (our condition on the integrable part is much more general than quasi-convexity). In the finitely differentiable case, we will obtain polynomial estimates of stability for perturbation of a generic integrable Hamiltonian.

\begin{theorem} \label{mainth2}
For $k>2n+2$, consider an arbitrary $C^k$ integrable Hamiltonian $h$ defined on an open ball in $\R^n$. Then for almost any $\xi\in\R^n$, the integrable Hamiltonian $h_\xi(x)=h(I)-\xi.I$ is polynomially stable.
\end{theorem}

Once again, this will be a direct consequence of Theorems~\ref{thmpreva} and Theorem~\ref{thmC} below, and the above result extends the main result of \cite{Bou10} which was only valid for quasi-convex integrable systems.

\section{Main results}

In order to state our results, we now describe our setting more precisely. We let $B_R=B(0,R)$ be the open ball of $\R^n$, centered at the origin, of radius $R>0$ with respect to the supremum norm $|\,.\,|$. Our phase space will be the domain $\mathcal{D}_R=\T^n \times B_R$.

\paraga Let us first explain our prevalent condition on the unperturbed Hamiltonian $h$, which comes from \cite{BN09}. Let $G(n,k)$ be the Grassmannian of all vector subspaces of $\R^n$ of dimension $k$. We equip $\R^n$ with the Euclidean scalar product, $\Vert\,.\,\Vert$ stands for the Euclidean norm, and given an integer $L\in\N^*$, we define $G^{L}(n,k)$ as the subset of $G(n,k)$ consisting of those subspaces whose orthogonal complement can be spanned by vectors $k\in\Z^n\setminus\{0\}$ with $|k|_1\leq L$, where $|\,.\,|_1$ is the $\ell^1$-norm.   

\begin{definition} \label{sdm}
A function $h\in C^2(B_R)$ is said to be Diophantine Morse if there exist $\gamma>0$ and $\tau \geq 0$ such that for any $L\in\N^*$, any $k\in\{1,\dots,n\}$ and any $\Lambda\in G^{L}(n,k)$, there exists $\left(e_1,\dots,e_k\right)$ (resp. $\left(f_1,\ldots,f_{n-k}\right)$), an orthonormal basis of $\Lambda$ (resp. of $\Lambda^\perp$), such that the function $h_\Lambda$ defined on $B_R$ by 
\[ h_\Lambda(\alpha,\beta)=h\left(\alpha_1 e_1+\dots+\alpha_k e_k+\beta_1f_1+\dots+\beta_{n-k} f_{n-k}\right), \]
satisfies the following: for any $(\alpha,\beta) \in B_R$, 
\[ \Vert\partial_\alpha h_\Lambda(\alpha,\beta)\Vert \leq \gamma L^{-\tau} \Longrightarrow \Vert \partial_{\alpha\alpha} h_\Lambda(\alpha,\beta).\eta\Vert>\gamma L^{-\tau}\Vert\eta\Vert \]
for any $\eta \in \R^n\setminus\{0\}$.
\end{definition} 

In other words, for any $(\alpha,\beta) \in B_R$, we have the following alternative: either $\Vert \partial_\alpha h_\Lambda(\alpha,\beta)\Vert>\gamma L^{-\tau}$ or $\Vert\partial_{\alpha\alpha} h_\Lambda(\alpha,\beta).\eta\Vert>\gamma L^{-\tau}\Vert\eta\Vert$ for any $\eta \in \R^n\setminus\{0\}$. This technical definition is basically a quantitative transversality condition which is stated in adapted coordinates. It is inspired on the one hand by the steepness condition introduced by Nekhoroshev (\cite{Nek77}) where one has to look at the projection of the gradient map $\nabla h$ onto affine subspaces, and on the other hand by the quantitative Morse-Sard theory of Yomdin (\cite{Yom83}, \cite{YC04}) where critical or ``nearly-critical" points of $h$ have to be quantitatively non degenerate. It is in fact equivalent to the condition introduced by Niederman in \cite{Nie07}: there the author considered the subset $G_{L}(n,k)$ of $G(n,k)$ consisting of those subspaces which can be spanned by vectors $k\in\Z^n\setminus\{0\}$ with $|k|_1\leq L$, but one can check that $G^{L}(n,k)$ is included in $G_{L^k}(n,k)$ (similarly, $G_{L}(n,k)$ is included in $G^{L^k}(n,k)$). Hence we will stick with the terminology ``Diophantine Morse" introduced in \cite{Nie07}, and the equivalent term ``Simultaneous Diophantine Morse" introduced in \cite{BN09} will not be used any more.

The set of Diophantine Morse functions on $B_R$ with respect to $\gamma>0$ and $\tau \geq 0$ will be denoted by $DM_{\gamma}^{\tau}(B_R)$, and we will also use the notations 
\[ DM^{\tau}(B_R)=\bigcup_{\gamma>0}DM_{\gamma}^{\tau}(B_R), \quad  DM(B_R)=\bigcup_{\tau \geq 0}DM^{\tau}(B_R). \]
We recall the following two results from \cite{Nie07} (see also \cite{BN09}).

\begin{theorem}\label{thmpreva}
Let $\tau>2(n^2+1)$ and $h\in C^{2n+2}(B_R)$. Then for Lebesgue almost all $\xi \in \R^n$, the function $h_\xi(I)=h(I)-\xi.I$ belongs to $DM^{\tau}(B_R)$.  
\end{theorem}

We already mentioned that there is a good notion of ``full measure" in an infinite dimensional vector space, which is called prevalence (see \cite{OY05} and \cite{HK10} for nice surveys), and the previous theorem has the following immediate corollary.

\begin{corollary}\label{corpreva}
For $\tau>2(n^2+1)$, $DM^{\tau}(B_R)$ is prevalent in $C^{2n+2}(B_R)$.  
\end{corollary}

\paraga Now let us introduce our regularity assumption, starting with the Gevrey case. Given $\alpha \geq 1$ and $L>0$, a real-valued function $H\in C^{\infty}(\mathcal{D}_R)$ is $(\alpha,L)$-Gevrey if, using the standard multi-index notation, we have
\[ |H|_{G^{\alpha,L}(\mathcal{D}_R)}=\sum_{l\in \N^{2n}}L^{|l|\alpha}(l!)^{-\alpha}|\partial^l H|_{C^{0}(\mathcal{D}_R)} < \infty \]
where $|\,.\,|_{C^{0}(\mathcal{D}_R)}$ is the usual supremum norm for functions on $\mathcal{D}_R$. The space of such functions, with the above norm, is a Banach algebra that we denote by $G^{\alpha,L}(\mathcal{D}_R)$, and in the sequel we shall simply write $|\,.\,|_{\alpha,L}=|\,.\,|_{G^{\alpha,L}(\mathcal{D}_R)}$. Analytic functions are a particular case of Gevrey functions, as one can check that $G^{1,L}(\mathcal{D}_R)$ is exactly the space of bounded real-analytic functions on $\mathcal{D}_R$ which extend as bounded holomorphic functions on the complex domain
\[ V_L(\mathcal{D}_{R})=\{(\theta,I)\in(\C^n/\Z^n)\times \C^{n} \; | \; |\mathcal{I}(\theta)|<L,\;d(I,B_R)<L\}, \] 
where $\mathcal{I}(\theta)$ is the imaginary part of $\theta$, $|\,.\,|$ the supremum norm on $\C^n$ and $d$ the associated distance on $\C^n$.

\paraga Therefore we shall consider a Hamiltonian
\begin{equation}\label{HamG}
\begin{cases} \tag{$\ast$}
H(\theta,I)=h(I)+f(\theta,I), \quad (\theta,I)\in \mathcal{D}_R, \\
|h|_{\alpha,L}=1, \; |f|_{\alpha,L} < \varepsilon. 
\end{cases}
\end{equation}
Our main result in the Gevrey case is the following.

\begin{theorem}\label{thmG}
Let $H$ be as in (\ref{HamG}), and assume that the integrable part $h$ belongs to $DM_{\gamma}^{\tau}(B)$, with $\tau\geq 2$ and $\gamma\leq 1$. Let us define
\[ a=b=3^{-1}(2(n+1)\tau)^{-n}.\]
Then there exists a constant $\varepsilon_0>0$, depending on $n, R, L, \alpha, \gamma$ and $\tau$, such that if $\varepsilon \leq \varepsilon_0$, for every initial action $I(0) \in B_{R/2}$ the following estimates
\[ |I(t)-I(0)| < (n+1)^2\varepsilon^b, \quad |t| \leq \exp(\varepsilon^{-\alpha^{-1}a}), \]
hold true.
\end{theorem}

For $\alpha=1$, we exactly recover the main theorem of \cite{BN09} including the value of the exponents, therefore the latter result is generalized to the Gevrey classes. Moreover, quasi-convex Hamiltonians are a very particular case of our class of Morse Diophantine Hamiltonians, hence the stability result of \cite{MS02} is also generalized, but not with the same exponents (we did not try to improve our exponents). Let us now explain several consequences of our result.

First, note that the only property used on the integrable part $h$ to derive these estimates is a specific steepness property, therefore the proof is also valid assuming a Diophantine steepness condition as in \cite{Nie07}, which is much more general than the original steepness condition of Nekhoroshev. Indeed, the class of Diophantine Morse functions (and {\it a fortiori} the class of Diophantine steep functions) contains fairly degenerate Hamiltonians, as for instance linear Hamiltonians with a Diophantine frequency, which of course are far from being steep. As a direct consequence, our main theorem also gives an alternative proof of exponential stability in the neighbourhood of a Gevrey Lagrangian quasi-periodic invariant torus, a fact which was only recently proved by Mitev and Popov in \cite{MP10} by the construction a Gevrey Birkhoff normal form.

Then, as it was proved in \cite{Bou09}, the method we are using is relatively intrinsic and does not depend much on the choice of coordinates. This remark is particularly useful when studying the stability in the neighbourhood of an elliptic fixed point, and more generally in the neighbourhood of a linearly stable lower-dimensional torus, under the common assumptions of isotropicity and reducibility (which are automatic for a fixed point or a Lagrangian torus). As in \cite{Bou09}, one can easily prove results of exponential stability in the Gevrey case under an appropriate Diophantine condition, therefore extending the results of exponential stability obtained in \cite{Bou09} which were valid in the analytic case. This also gives an extension of the stability result of \cite{MP10} which is only available for a Gevrey Lagrangian torus.

In \cite{Bou09}, using the idea introduced by Morbidelli and Giorgilli (\cite{MG95}) to combine Birkhoff normal forms and Nekhoroshev estimates, we also had results of super-exponential stability under a Diophantine condition on the frequency and a prevalent condition on the formal series of Birkhoff invariants. Using the Gevrey Birkhoff normal form of \cite{MP10}, we can also extend this super-exponential stability result to Gevrey classes, but only for a Lagrangian torus. For a more general linearly stable torus (isotropic, reducible), the existence of a Gevrey Birkhoff normal form undoubtedly holds true but it is still missing.   

As a last remark, we would like to point out that one can also extend a fairly different result of stability, which is due to Berti, Bolle and Biasco (\cite{BBB03}). This concerns perturbations of {\it a priori} unstable Hamiltonians systems, which have intensely studied since instability properties in this context are much more simple to exhibit. In the analytic case, if the size of the perturbation is $\mu$, it was proved in \cite{BBB03} that the optimal time of instability is $\tau(\mu) \simeq \mu^{-1}\ln\mu^{-1}$. The upper bound $\tau(\mu) \lesssim \mu^{-1}\ln\mu^{-1}$ follows from a specific construction of an unstable solution, while the lower bound $\tau(\mu) \gtrsim \mu^{-1}\ln\mu^{-1}$ was a consequence of a stability result, where the analyticity of the system was only necessary to apply Nekhoroshev estimates both in the quasi-convex and steep case on certain regions of the phase space (but because of the presence of ``hyperbolicity", the global stability time is far from being exponentially large). In \cite{BP10}, we introduced yet another technique (which pertains more to dynamical systems, as opposed to the variational arguments of \cite{BBB03}) to construct a solution for which $\tau(\mu) \lesssim \mu^{-1}\ln\mu^{-1}$, but only the Gevrey case. Now having at our disposal Nekhoroshev estimates in the Gevrey case for both quasi-convex and steep integrable systems, this implies that the lower bound $\tau(\mu) \gtrsim \mu^{-1}\ln\mu^{-1}$ can also be obtained and that the time of instability $\tau(\mu) \simeq \mu^{-1}\ln\mu^{-1}$ is also optimal in the Gevrey case (in fact, using Theorem~\ref{thmC} below, this is also true if the system is $C^k$ for $k$ large enough). This justify the optimality we claimed in \cite{BP10}. 

\paraga Let us now explain our result in the finitely differentiable case. Here we assume that $H$ is of class $C^k$, {\it i.e.} it is $k$-times differentiable and all its derivatives up to order $k$ extend continuously to the closure $\overline{\mathcal{D}}_R$. In order to have non-trivial results, we shall assume a minimal amount of regularity, that is $k\geq n+1$ and it will convenient to introduce another parameter of regularity $k^*\in\N^*$ satisfying $k\geq k^*n+1$. We denote by $C^k(\mathcal{D}_R)$ the space of functions of class $C^k$ on $\mathcal{D}_R$, which is a Banach algebra with the norm
\[ |H|_{C^k(\mathcal{D}_R)}=\sum_{|l|\leq k}(l!)^{-1}|\partial^l H|_{C^{0}(\mathcal{D}_R)}, \]
where we have used the standard multi-index notation and where $|\,.\,|_{C^{0}(\mathcal{D}_R)}$ still denotes the usual supremum norm for functions on $\mathcal{D}_R$. Once again, for simplicity, we shall only write $|\,.\,|_{k}=|\,.\,|_{C^k(\mathcal{D}_R)}$.

\paraga So now we consider a Hamiltonian of the form
\begin{equation}\label{HamC}
\begin{cases} \tag{$\ast\ast$}
H(\theta,I)=h(I)+f(\theta,I), \quad (\theta,I)\in \mathcal{D}_R, \\
|h|_{k}=1, \; |f|_{k} < \varepsilon. 
\end{cases}
\end{equation}
Our main result in the finitely differentiable case is the following.

\begin{theorem}\label{thmC}
Let $H$ be as in (\ref{HamC}), assume that the integrable part $h$ belongs to $DM_{\gamma}^{\tau}(B)$, with $\tau\geq 2$ and $\gamma\leq 1$, and that $k\geq k^*n+1$ for some $k^*\in\N^*$. Let us define
\[ a=b=3^{-1}(2(n+1)\tau)^{-n}.\]
Then there exists a constant $\varepsilon_0>0$, depending on $n, R, k, \gamma$ and $\tau$, such that if $\varepsilon \leq \varepsilon_0$, for every initial action $I(0) \in B_{R/2}$ the following estimates
\[ |I(t)-I(0)| < (n+1)^2\varepsilon^b, \quad |t| \leq \varepsilon^{-k^*a}, \]
hold true.
\end{theorem}

The above theorem extends the main result of \cite{Bou10}, which was only valid for quasi-convex integrable Hamiltonians. Let us point out that in this result (as in the one contained in \cite{Bou10}) we have decided to consider only the case of integer values of $k$, but the results can also be extended to real values (that is, to Hölder spaces) and this would have given a more precise exponent of stability in terms of the regularity of the system, but we decided not to pursue this further.

\paraga Let us now conclude with some notations that we shall use throughout the text. 

First, we have define norms for Gevrey and $C^k$ functions, but we shall need corresponding norms for vector-valued functions (in particular for diffeomorphisms). Hence given a vector-valued function $F : \mathcal{D}_R \rightarrow \R^{m}$, $m\in\N^*$ and $F=(F_1,\dots,F_{m})$, we say that $F$ is $(\alpha,L)$-Gevrey if $F_i \in G^{\alpha,L}(\mathcal{D}_R)$, for $1\leq i \leq m$, and we will write $|F|_{\alpha,L}=\sum_{i=1}^{m}|F_i|_{\alpha,L}$. Similarly, $F$ is of class $C^k$ if $F_i \in C^k(\mathcal{D}_R)$, for $1\leq i \leq m$, and we will write $|F|_{k}=\sum_{i=1}^{m}|F_i|_{k}$. 

Then, to avoid cumbersome expressions, we will replace constants depending only on $n, R, L, \alpha, \gamma$ and $\tau$ (resp. on $n, R, k, \gamma$ and $\tau$) in the Gevrey case (resp. in the $C^k$ case) with a dot. More precisely, an assertion of the form ``there exists a constant $c\geq1$ depending on the above parameters such that $u < cv$'' will be simply replaced with ``$u\MP v$'', when the context is clear.

\section{Analytical part}\label{ana}

In this part, we shall describe and prove some normal forms that we will need for the proofs of Theorem~\ref{thmG} and Theorem~\ref{thmC}. More precisely, Gevrey Hamiltonians will be considered in section~\ref{ana1} and finitely differentiable Hamiltonians in section~\ref{ana2}, and eventually in section~\ref{ana3} we will explain the dynamical consequences of these normal forms.    

But first we need to recall the following basic definition, which will be crucial to us.

\begin{definition}
A vector $\omega \in \R^n\setminus\{0\}$ is said to be periodic if there exists a real number $t>0$ such that $t\omega \in \Z^n$. In this case, the number 
\[ T=\inf\{t>0 \; | \; t\omega \in \Z^n\}\] 
is called the period of $\omega$.
\end{definition}

The easiest example is given by a vector with rational components, the period of which is just the least common multiple of the denominators of its components. Geometrically, if $\omega$ is $T$-periodic, an invariant torus with a linear flow with vector $\omega$ is filled with $T$-periodic orbits.

\subsection{The Gevrey case}\label{ana1}

As in \cite{MS02}, we shall start with perturbations of linear integrable Hamiltonians, for which we will obtain a global normal form (Lemma~\ref{lemmelincomp} below), and then the latter will be used to obtain local normal forms for perturbations of general Hamiltonians (Proposition~\ref{propcomp} below).

\paraga Let $\omega_1\in\R^n\setminus\{0\}$ be a $T_1$-periodic vector, and let $l_1(I)=\omega_1.I$ be the linear integrable Hamiltonian with frequency $\omega_1$. In the following, we shall consider  a ``large" positive integer $m\in\N^*$ and a ``small" parameter $\mu_1>0$, which will eventually depend on $\varepsilon$. We shall also use a real number $\rho_1>0$ independent of $\varepsilon$, to be fixed below. The following result is due to Marco and Sauzin (\cite{MS02}).

\begin{lemma}[Marco-Sauzin]\label{lemmelin}
Consider the Hamiltonian $H=l_1+f$ defined on $\mathcal{D}_{3\rho_1}$, with $f\in G^{\alpha,L}(\mathcal{D}_{3\rho_1})$ and $|f|_{\alpha,L}<\mu_1$. Assume that
\begin{equation}\label{hyplin}
T_1\mu_1 \PM 1, \quad mT_1\mu_1 \PM 1.
\end{equation}
Then there exist $L_1=CL$, for some constant $0<C<1$, and an $(\alpha,L_1)$-Gevrey symplectic transformation
\[ \Phi^1 : \mathcal{D}_{2\rho_1} \rightarrow \mathcal{D}_{3\rho_1} \]
with $|\Phi^1-\mathrm{Id}|_{\alpha,L_1} \MP T_1\mu_1$ such that
\[ H^1=H\circ\Phi^1=l_1+g^1+f^1 \]
with $\{g^1,l_1\}=0$ and the estimates
\[ |g^1|_{\alpha,L_1} \MP \mu_1, \quad |f^1|_{\alpha,L_1} \MP e^{-m^{1/\alpha}}\mu_1,  \]
hold true.
\end{lemma}

One can choose the constant $C=16^{-1}(2n)^{\frac{1-\alpha}{\alpha}}$. The statement above is exactly Proposition 3.2 in \cite{MS02}, where the authors state their result for $mT_1\mu_1\PE 1$, but it also holds trivially for smaller $m$, that is $mT_1\mu_1\PM 1$. The use of this artificial parameter $m\in\N^*$ will make subsequent arguments easier. 

It is perhaps useful to understand this lemma in the very special case where $\omega_1=e_1$ is the first vector of the canonical basis of $\R^n$: the equality $\{g^1,l_1\}=0$ simply means that $g^1$ is independent of the first angle $\theta_1$, and therefore the evolution of the first action component $I_1$ is only governed by the remainder $f^1$. 

\paraga Now we are going to make suitable ``compositions" of the above lemma, but first we shall explain heuristically what we are planning to do formally in the sequel. 

So we consider another periodic vector $\omega_2\in\R^n\setminus\{0\}$, with period $T_2>0$, which is independent of $\omega_1$, and we let $l_2(I)=\omega_2.I$. If the suitable hypotheses are met, by Lemma~\ref{lemmelin} we can transform $H=l_2+f$, where the size of $f$ is of order $\mu_2$, into 
\[ H^2=H\circ \Phi^2=l_2+g^2+f^2, \quad \{g^2,l_2\}=0,\] 
with $g^2$ of order $\mu_2$ and $f^2$ of order $e^{-m}\mu_2$. Now if $\omega_2$ is close enough to $\omega_1$, that is $|\omega_2-\omega_1|<\mu_1$, and if $\mu_2<\mu_1$, we can write \[ H^2=(l_2+l_1-l_1)+g^2+f^2=l_1+(l_2-l_1+g^2)+f^2=l_1+\tilde{f}+f^2, \] 
where $\tilde{f}=l_2-l_1+g^2$ satisfies $\{\tilde{f},l_2\}$ and its size is of order $\mu_1$. For a moment, let us forget about $f^2$, which is already exponentially small with respect to $m$, and consider $l_1+\tilde{f}$ as a perturbation of $l_1$. Under the suitable assumptions, we can apply once again Lemma~\ref{lemmelin} and find a transformation $\Phi^1$ that sends $l_1+\tilde{f}$ into $l_1+g^1+f^1$, where $f^1$ is exponentially small with respect to $m$ and $\{g^1,l_1\}=0$.

Now the key point is the following: as $\tilde{f}$ satisfies $\{\tilde{f},l_2\}=0$, $g^1$ and $f^1$ also satisfy $\{g^1,l_2\}=\{f^1,l_2\}=0$, hence $\{g^1,l_1\}=\{g^1,l_2\}=0$. Indeed, it is enough to show that, denoting $\left(\Phi_{s}^{l_1}\right)_{s\in\R}$ the Hamiltonian flow of $l_1$, if $\{\tilde{f},l_2\}=0$ then   
\[ [\tilde{f}]_1=\frac{1}{T_1}\int_{0}^{T_1}\tilde{f}\circ\Phi_{s}^{l_1}ds \]
and 
\[ \chi_1=\frac{1}{T_1}\int_{0}^{T_1}(\tilde{f}-[\tilde{f}]_1)\circ\Phi_{s}^{l_1}sds \]
also satisfy $\{[\tilde{f}]_1,l_2\}=0$ and $\{\chi_1,l_2\}=0$. This can be easily proved by direct computations, but this is a general fact in normal form theory (sometimes known as a ``normal form with symmetry") and a nicer way to see this goes as follows. Since $\{l_1,l_2\}=0$, the linear operators $L_{l_1}=\{.,l_1\}$ and $L_{l_2}=\{.,l_2\}$ commutes, so that the kernel of $L_{l_2}$ is invariant by $L_{l_1}$, and as $L_{l_1}$ is semi-simple, the kernel of $L_{l_2}$ is also invariant under the projection onto the kernel of $L_{l_1}$ which is given by the map $[\,.\,]_1$. This explains why $\{[\tilde{f}]_1,l_2\}=0$. Now $\tilde{f}-[\tilde{f}]_1$ is in the kernel of $L_{l_2}$, and its unique pre-image by $L_{l_1}$ is given by $\chi_1$, hence $\{\chi_1,l_2\}=0$. Put it differently, if a Hamiltonian (above $l_1+\tilde{f}$) has an integral (in our case, $l_2$), then the integral is invariant under the normalizing transformation (in our case, $l_2 \circ \Phi_1=l_2$) and $l_2$ remains an integral of the normalized Hamiltonian (that is $\{l_1+g^1+f^1,l_2\}$).  

To conclude, taking into account $f^2$, the map $\Phi^1$ sends $f^2$ to $f^2\circ\Phi^1$ which remains exponentially small with respect to $m$, and so is $f_2=f^1+f^2\circ\Phi^1$. Therefore   
\[ H^2\circ\Phi^1=l_1+g^1+f_2, \quad \{g^1,l_1\}=\{g^1,l_2\}=0, \]
then setting $g_2=l_1-l_2+g^1$, we can write again
\[ H^2\circ\Phi^1=l_2+g_2+f_2, \quad \{g_2,l_1\}=\{g_2,l_2\}=0. \]
Finally, we have found $\Phi_2=\Phi^2\circ\Phi^1$ such that
\[ H_2=H\circ\Phi_2=H^2\circ\Phi^1=H\circ\Phi^2\circ\Phi^1=l_2+g_2+f_2 \]
with $\{g_2,l_1\}=\{g_2,l_2\}=0$ and $f_2$ exponentially small.

\paraga Now let us make our previous discussion rigorous.

For $i\in\{1,\dots,n\}$, let $\omega_i\in\R^n\setminus\{0\}$ be independent $T_i$-periodic vectors, and we denote by $l_i$ the linear integrable Hamiltonian of frequency $\omega_i$. We consider a positive integer $m\in\N^*$ and a sequence of small parameters $\mu_i>0$, for $i\in\{1,\dots,n\}$. As before, $m\in\N^*$ and $\mu_i>0$, for $i\in\{1,\dots,n\}$, will eventually depend on $\varepsilon$. Now to fix the ideas, we define the increasing sequence
\[ \rho_i=2^{i}, \quad i\in\{1,\dots,n\}. \]

We shall need some assumptions on these parameters, so we define the condition $(A_i)$ for $i\in\{1,\dots,n\}$ by
\begin{equation}\label{A1}
T_1\mu_1 \PM 1, \quad mT_1\mu_1 \PM 1 \tag{$A_1$}
\end{equation} 
and for $i\in\{2,\dots,n\}$,
\begin{equation}\label{Ai}
T_i\mu_i \PM 1, \; mT_i\mu_i \PM 1, \;|\omega_i-\omega_{i-1}|\MP\mu_{i-1}, \; \mu_i\PM\mu_{i-1}. \tag{$A_i$}
\end{equation}

Recalling the constant $C>0$ that appeared in Lemma~\ref{lemmelin}, we shall also define the decreasing sequence $L_i=C^iL$, for $i\in\{0,\dots,n\}$. In the above lemma, we shall use Lemma~\ref{estimcomp} of appendix~\ref{tech}.

\begin{lemma}\label{lemmelincomp}
Let $j\in\{1,\dots,n\}$, and consider the Hamiltonian $H=l_j+f$ defined on $\mathcal{D}_{3\rho_j}$, with $f\in G^{\alpha,L}(\mathcal{D}_{3\rho_j})$ and $|f|_{\alpha,L}\MP\mu_j$. Assume that $(A_i)$ is satisfied for $i\in\{1,\dots,j\}$.
Then there exists a symplectic transformation
\[ \Phi_j=\Phi^j \circ \cdots \circ \Phi^1 : \mathcal{D}_{2\rho_1} \rightarrow \mathcal{D}_{3\rho_j} \]
where $\Phi^i : \mathcal{D}_{2\rho_i} \rightarrow \mathcal{D}_{3\rho_i}$ is $(\alpha,L_{j-i+1})$-Gevrey and $|\Phi^i-\mathrm{Id}|_{\alpha,L_{j-i+1}} \MP T_i\mu_i$ for $i\in\{1,\dots,j\}$, such that
\[ H_j=H\circ\Phi_j=l_j+g_j+f_j \in G^{\alpha,L_j}(\mathcal{D}_{2\rho_1}) \]
with $\{g_j,l_i\}=0$, for $i\in\{1,\dots,j\}$, and the estimates
\[ |g_j|_{\alpha,L_j} \MP \mu_1, \quad |f_j|_{\alpha,L_j} \MP e^{-m^{1/\alpha}}\mu_1,  \]
hold true.
\end{lemma}

\begin{proof}
The proof goes by induction. For $j=1$, this is nothing but Lemma~\ref{lemmelin} with $g_1=g^1$ and $f_1=f^1$. Now assume that the result holds true for some $j-1\in\{1,\dots,n-1\}$, and let us show that it remains true for $j\in\{2,\dots,n\}$.

By assumption, $(A_j)$ is satisfied, in particular
\[ T_j\mu_j \PM 1, \quad m T_j\mu_j \PM 1 \]
hence condition~(\ref{hyplin}) of Lemma~\ref{lemmelin} holds true. Therefore, there exists an $(\alpha,L_1)$-Gevrey symplectic transformation
\[ \Phi^j : \mathcal{D}_{2\rho_j} \rightarrow \mathcal{D}_{3\rho_j} \]
with $|\Phi^j-\mathrm{Id}|_{\alpha,L_1} \MP T_j\mu_j$ such that
\[ H^j=H\circ\Phi^j=l_j+g^j+f^j \]
with $\{g^j,l_j\}=0$ and the estimates
\[ |g^j|_{\alpha,L_1} \MP \mu_j, \quad |f^j|_{\alpha,L_1} \MP e^{-m^{1/\alpha}}\mu_j,  \]
hold true.

Now let us introduce $\tilde{f}=l_j-l_{j-1}+g^j$. Obviously, we have $\{\tilde{f},l_j\}=0$. Moreover, by assumption $(A_j)$ we have $|\omega_j-\omega_{j-1}|\MP\mu_{j-1}$ and $\mu_j\PM\mu_{j-1}$ so that 
\begin{eqnarray*}
|\tilde{f}|_{\alpha,L_1} & \leq & |l_j-l_{j-1}|_{\alpha,L_1} + |g^j|_{\alpha,L_1} \\
& \MP & |\omega_j-\omega_{j-1}| + \mu_j \\
& \MP & \mu_{j-1}.
\end{eqnarray*}  
Then we can write
\[ H^j=l_{j-1}+\tilde{f}+f^j. \]
Furthermore, as $\mathcal{D}_{3\rho_{j-1}}\subseteq\mathcal{D}_{2\rho_j}$, the Hamiltonian $l_{j-1}+\tilde{f}$ is well-defined on $\mathcal{D}_{3\rho_{j-1}}$, we have $\tilde{f}\in G^{\alpha,L_1}(\mathcal{D}_{3\rho_{j-1}})$ with $|\tilde{f}|_{\alpha,L_1}\MP\mu_{j-1}$. 

Now recall that $(A_i)$ holds true for $i\in\{1,\dots,j-1\}$, hence we can eventually apply our hypothesis of induction to the Hamiltonian $l_{j-1}+\tilde{f}$: there exists a symplectic transformation
\[ \Phi_{j-1}=\Phi^{j-1} \circ \cdots \circ \Phi^1 : \mathcal{D}_{2\rho_{1}} \rightarrow \mathcal{D}_{3\rho_{j-1}} \]
where $\Phi^i : \mathcal{D}_{2\rho_i} \rightarrow \mathcal{D}_{3\rho_i}$ is $(\alpha,L_{j-i+1})$-Gevrey and $|\Phi^i-\mathrm{Id}|_{\alpha,L_{j-i+1}} \MP T_i\mu_i$ for $i\in\{1,\dots,j-1\}$, such that
\[ (l_{j-1}+\tilde{f})\circ\Phi_{j-1}=l_{j-1}+g_{j-1}+f_{j-1} \in G^{\alpha,L_{j-1}}(\mathcal{D}_{2\rho_1}) \]
with $\{g_{j-1},l_i\}=0$, for $i\in\{1,\dots,j-1\}$, and the estimates
\[ |g_{j-1}|_{\alpha,L_{j-1}} \MP \mu_1, \quad |f_{j-1}|_{\alpha,L_{j-1}} \MP e^{-m^{1/\alpha}}\mu_1,  \]
hold true. Moreover, as $\{\tilde{f},l_j\}=0$, we also have $\{g_{j-1},l_j\}=0$ and therefore        $\{g_{j-1},l_i\}=0$, for $i\in\{1,\dots,j\}$. 

Then we set $\Phi_j=\Phi^j\circ\Phi_{j-1}$ so that 
\[ \Phi_j=\Phi^j\circ \cdots \circ \Phi^1 : \mathcal{D}_{2\rho_{1}} \rightarrow \mathcal{D}_{3\rho_{j}}. \]
Now
\begin{eqnarray*}
H_j & = & H\circ\Phi_j\\
& = & H^j\circ\Phi_{j-1} \\
& = & (l_{j-1}+\tilde{f})\circ \Phi_{j-1} + f^j\circ \Phi_{j-1}\\
& = & l_{j-1}+g_{j-1}+f_{j-1}+f^j\circ \Phi_{j-1}.
\end{eqnarray*}
We will prove below that $f^j\circ \Phi_{j-1} \in G^{\alpha,L_j}(\mathcal{D}_{2\rho_1})$, and since we know that the function $l_{j-1}+g_{j-1}+f_{j-1} \in G^{\alpha,L_{j-1}}(\mathcal{D}_{2\rho_1})$, this will easily implies that $H_j \in G^{\alpha,L_j}(\mathcal{D}_{2\rho_1})$. Now let us define $g_j=l_{j-1}-l_j+g_{j-1}$ and $f_j=f_{j-1}+f^j\circ \Phi_{j-1}$ so that we can eventually write
\[ H_j=l_j+g_j+f_j. \]
Since $\{g_{j-1},l_i\}=0$ then obviously $\{g_j,l_i\}=0$, for $i\in\{1,\dots,j\}$. Therefore it remains to prove the estimates. First, we have
\begin{eqnarray*}
|g_j|_{\alpha,L_j} & \leq & |l_{j-1}-l_j|_{\alpha,L_j} + |g_{j-1}|_{\alpha,L_{j}} \\
& \leq & |l_{j-1}-l_j|_{\alpha,L_j} + |g_{j-1}|_{\alpha,L_{j-1}} \\
& \MP & |\omega_{j-1}-\omega_j| + \mu_1 \\
& \MP & \mu_{j-1}+ \mu_1 \\
& \MP & \mu_1.
\end{eqnarray*}  
Then, we know that $|f_{j-1}|_{\alpha,L_{j-1}} \MP e^{-m^{1/\alpha}}\mu_1$, so we only need to estimate $|f^j\circ \Phi_{j-1}|_{\alpha,L_j}$. For that, recall that
\[ \Phi_{j-1}=\Phi^{j-1} \circ \cdots \circ \Phi^{1} \]
and for $i\in\{1,\dots,j-1\}$, we have the estimates
\[ |\Phi_i-\mathrm{Id}|_{\alpha,L_{j-i+1}} \MP T_i\mu_i \MP 1. \]
Therefore a repeated use of lemma~\ref{estimcomp} yields 
\begin{eqnarray*}
|f^j\circ \Phi_{j-1}|_{\alpha,L_j} & = & |f^j \circ \Phi^{j-1} \circ \cdots \circ \Phi^{1}|_{\alpha,L_{j}} \\ 
& \leq & |f^j \circ \Phi^{j-1} \circ \cdots \circ \Phi^{2}|_{\alpha,L_{j-1}} \\
& \leq & |f^j \circ \Phi^{j-1} \circ \cdots \circ \Phi^{3}|_{\alpha,L_{j-2}} \\
& \dots & \\
& \leq & |f^j|_{\alpha,L_1} \\
& \MP & e^{-m}\mu_j. 
\end{eqnarray*}
Hence
\begin{eqnarray*}
|f_j|_{\alpha,L_j} & \leq & |f_{j-1}|_{\alpha,L_{j}}+ |f^j\circ \Phi_{j-1}|_{\alpha,L_j}\\
& \leq & |f_{j-1}|_{\alpha,L_{j-1}}+ |f^j\circ \Phi_{j-1}|_{\alpha,L_j}\\
& \MP & e^{-m^{1/\alpha}}(\mu_1+\mu_j) \\
& \MP & e^{-m^{1/\alpha}}\mu_1, 
\end{eqnarray*}
which is the required estimate.
\end{proof}

Here also, it is perhaps useful to understand this lemma in the special case where $(\omega_1, \dots,\omega_n)$ is the canonical basis of $\R^n$: the equality $\{g_j,l_i\}=0$ for $i\in\{1,\dots,j\}$ means that $g_j$ is independent of the first $j$ angles $\theta_1, \dots, \theta_j$, and therefore the evolution of the first $j$ action components $I_1,\dots,I_j$ is only governed by the remainder $f^j$. In any cases, since we are assuming that $(\omega_1, \dots,\omega_n)$ are linearly independent, then for $j=n$, $g_n$ is integrable and the action variables can only evolve according to $f^n$.  

\paraga Now we shall come back to our original setting~(\ref{HamG}), that is 
\begin{equation*}
\begin{cases} 
H(\theta,I)=h(I)+f(\theta,I), \quad (\theta,I)\in \mathcal{D}_R, \\
|h|_{\alpha,L}=1, \; |f|_{\alpha,L} < \varepsilon. 
\end{cases}
\end{equation*}
For $i\in\{1,\dots,n\}$, we still consider a sequence of $T_i$-periodic vectors $\omega_i$, a sequence of small parameters $\mu_i$ and an integer $m\in\N^*$.

Let us fix $i\in\{1,\dots,n\}$. If we were able to find a $T_i$-periodic action $I_i\in B_R$ linked to $\omega_i$, that is satisfying $\nabla h(I_i)=\omega_i$, then on a small ball of radius $\mu_i$ around $I_i$, we could perform some standard scalings to reduce the study of perturbations of $h$ to the study of perturbations of the linear Hamiltonian $l_i(I)=\omega_i.I$, and so we could use the results of the previous section. However, in the sequel we will construct $\omega_i$, but since we are not assuming that the gradient map of $h$ is invertible, we cannot construct a corresponding action. In fact, this is not a serious problem, but this is just meant to explain why we need to use some slightly twisted arguments below. In \cite{BN09}, we used the idea, coming from \cite{Nie07}, to define domains directly in the space of frequencies, however this lead to a rather cumbersome definition of domains. Here we shall use a simpler approach that will enable us to work in the space of actions.

For $i\in\{1,\dots,n\}$, we consider a sequence of actions $I_i$ which is $\mu_i$-linked to the sequence of independent periodic vectors $\omega_i$, in the sense that
\[ |\nabla h(I_i)-\omega_i|<\mu_i. \]  
By the construction of our periodic vectors, such actions will indeed exist.

Taking into account this sequence of actions $(I_1,\dots,I_n)$ and the size of the perturbation $\varepsilon$, we define some new assumptions $(B_i)$, for $i\in\{1,\dots,n\}$, by 
\begin{equation} \label{B1}
\begin{cases}
T_1\mu_1 \PM 1, \; mT_1\mu_1 \PM 1,\\
\mu_1\PM 1, \; \varepsilon<\mu_1^2, \; |\nabla h(I_1)-\omega_1|<\mu_1
\end{cases} \tag{$B_1$}
\end{equation}
and for $i\in\{2,\dots,n\}$,
\begin{equation} \label{Bi}
\begin{cases}
T_i\mu_i \PM 1, \; mT_i\mu_i \PM 1, \;|\omega_i-\omega_{i-1}|\MP\mu_{i-1}, \; \mu_i\PM\mu_{i-1},\\
\mu_i\PM 1, \; \varepsilon<\mu_i^2, \; |\nabla h(I_i)-\omega_i|<\mu_i. 
\end{cases} \tag{$B_i$}
\end{equation}

In the proposition below, we shall denote by $\Pi_I : \mathcal{D}_R \rightarrow B_R$ the projection onto the action space, and we shall make use of Lemma~\ref{cauchy} in appendix~\ref{tech}. 

\begin{proposition}\label{propcomp}
Suppose $H$ is as in~(\ref{HamG}), and assume that $(B_i)$ is satisfied for $i\in\{1,\dots,j\}$. Then there exists a $C^{\infty}$ symplectic transformation 
\[\Psi_j :\T^n \times B(I_j,2\rho_1\mu_j)\rightarrow \T^n \times B(I_j,3\rho_j\mu_j)\] 
with $|\Pi_I\Psi_j-\mathrm{Id}_I|_{C^{0}(B(I_j,2\rho_1\mu_j))}\PM\mu_j$ such that 
\begin{equation*}
H \circ \Psi_j=h+g_j+f_j,
\end{equation*}
with $\{g_j,l_i\}=0$ for $i\in\{1,\dots,j\}$ and the estimate
\begin{equation*}
|\partial_{\theta} f_j|_{C^{0}(\T^n \times B(I_j,2\rho_1\mu_j))}<e^{-m^{1/\alpha}}\mu_j
\end{equation*} 
holds true. 
\end{proposition} 

Note that the proof gives in fact slightly better estimates, and also an estimate on the size of the function $g_j$, but this will not be needed in the following. The case $j=1$ is due to Marco and Sauzin (\cite{MS02}), and together with an estimate on $g_j$ this case is sufficient to prove effective stability for quasi-convex unperturbed systems. But here we shall need this result for any $j\in\{1,\dots,n\}$: indeed, since we are assuming the periodic vectors to be independent, the above proposition will give us information on resonances of any multiplicities (the case $j\in\{1,\dots,n\}$ corresponds to a resonance of multiplicity $n-j$). 

\begin{proof}
To analyze our Hamiltonian $H$ in a neighbourhood of size $\mu_j$ around $I_j$, we translate and rescale the action variables using the conformally symplectic map
\[ \sigma_j : (\theta,\tilde{I}) \longmapsto (\theta,I)=(\theta,I_j+\mu_j \tilde{I}) \]
which sends the domain $\mathcal{D}_{3\rho_j}=\T^n \times B_{3\rho_j}$ onto $\T^n \times B(I_j,3\rho_j\mu_j)$. By the condition $\mu_j \MP 1$ in $(B_j)$, we can assume that the latter domain is included in $\mathcal{D}_R$. Let 
\[ \tilde{H}=\mu_{j}^{-1}(H\circ\sigma_j)\] 
be the rescaled Hamiltonian, so $\tilde{H}$ is defined on $\mathcal{D}_{3\rho_j}$ and reads
\[ \tilde{H}(\theta,\tilde{I})=\mu_j^{-1}H(\theta,I_j+\mu_j \tilde{I})=\mu_j^{-1}h(I_j+\mu_j \tilde{I})+\mu_j^{-1}f(\theta,I_j+\mu_j \tilde{I}) \]
for $(\theta,\tilde{I})\in\ \mathcal{D}_{3\rho_j}$. Now using Taylor's formula we can expand $h$ around $I_j$ and, assuming with no loss of generality that $h(I_j)=0$, we obtain
\begin{eqnarray*}
h(I_j+\mu_j \tilde{I}) & = & \mu_j\nabla h(I_j).\tilde{I}+\mu_j^2\int_{0}^{1}(1-t)\nabla^2 h(I_j+t\mu_j \tilde{I})\tilde{I}.\tilde{I} dt \\
& = & \mu_j\omega_j.\tilde{I}+\mu_j(\nabla h(I_j)-\omega_j).\tilde{I} \\
& + & \mu_j^2\int_{0}^{1}(1-t)\nabla^2 h(I_j+t\mu_j \tilde{I})\tilde{I}.\tilde{I} dt \\
& = & \mu_j\omega_j.\tilde{I}+\mu_j\tilde{h}(\tilde{I}) 
\end{eqnarray*}
where we have defined
\[ \tilde{h}(\tilde{I})=(\nabla h(I_j)-\omega_j).\tilde{I}+\mu_j\int_{0}^{1}(1-t)\nabla^2 h(I_j+t\mu_j \tilde{I})\tilde{I}.\tilde{I} dt. \]
Therefore we can write
\[ \tilde{H}=l_j+\tilde{f} \]
with 
\[ \tilde{f}=\tilde{h}+\mu_j^{-1}(f\circ\sigma_j). \]
Let us estimate the norm of $\tilde{f}$. Setting $\tilde{L}=L/2$, by using Lemma~\ref{cauchy} (with $p=2$) we can bound the $(\alpha,\tilde{L})$-norm of $\nabla^2 h$ in terms of the $(\alpha,L)$-norm of $h$. Now recalling that $|h|_{\alpha,L}=1$, we obtain
\[ |\nabla^2 h|_{\alpha,\tilde{L}} \MP 1, \]
and as $|\nabla h(I_j)-\omega_j|<\mu_j$ (this is part of assumption $(B_j)$), we have  
\[ |\tilde{h}|_{\alpha,\tilde{L}} \MP \mu_j. \]
Then, by $(B_j)$ again, $\varepsilon<\mu_j^2$ hence
\[ |\mu_j^{-1}(f\circ\sigma_j)|_{\alpha,L} \leq \mu_j^{-1}\varepsilon < \mu_j. \]
and therefore
\[ |\tilde{f}|_{\alpha,\tilde{L}} \leq |\tilde{h}|_{\alpha,\tilde{L}}+|\mu_j^{-1}(f\circ\sigma_j)|_{\alpha,L}\MP\mu_j. \]   
As $(B_i)$ implies $(A_i)$ for $i\in\{1,\dots,j\}$, we can eventually apply Lemma~\ref{lemmelincomp} to the Hamiltonian $\tilde{H}=l_j+\tilde{f}$ (replacing $L$ by $\tilde{L}$ and $L_j$ by $\tilde{L}_j$): there exists a symplectic transformation
\[ \tilde{\Phi}_j=\tilde{\Phi}^j \circ \cdots \circ \tilde{\Phi}^1 : \mathcal{D}_{2\rho_1} \rightarrow \mathcal{D}_{3\rho_j} \]
where $\tilde{\Phi}^i : \mathcal{D}_{2\rho_i} \rightarrow \mathcal{D}_{3\rho_i}$ is $(\alpha,\tilde{L}_{j-i+1})$-Gevrey and $|\tilde{\Phi}^i-\mathrm{Id}|_{\alpha,\tilde{L}_{j-i+1}} \MP T_i\mu_i$ for $i\in\{1,\dots,j\}$, such that
\[ \tilde{H}_j=\tilde{H}\circ\tilde{\Phi}_j=l_j+\tilde{g}_j+\tilde{f}_j \in G^{\alpha,\tilde{L}_j}(\mathcal{D}_{2\rho_1}) \]
with $\{\tilde{g}_j,l_i\}=0$, for $i\in\{1,\dots,j\}$, and the estimates
\[ |\tilde{g}_j|_{\alpha,\tilde{L}_j} \MP \mu_1, \quad |\tilde{f}_j|_{\alpha,\tilde{L}_j} \MP e^{-m^{1/\alpha}}\mu_1,  \]
hold true. Moreover, if we introduce
\[ \tilde{s}_j=\tilde{g}_j-\tilde{h},\] 
we still have $\{\tilde{s}_j,l_j\}=0$, for $i\in\{1,\dots,j\}$, and so the transformed Hamiltonian can also be written as
\begin{equation*}
\tilde{H}_j=\tilde{H}\circ\tilde{\Phi}_j=l_j+\tilde{h}+\tilde{s}_j+\tilde{f}_j.
\end{equation*}
Now scaling back to our original coordinates, we define $\Psi_j=\sigma_{j} \circ \tilde{\Phi}_j \circ \sigma_{j}^{-1}$, therefore
\[ \Psi_j : \T^n\times B(I_j,2\rho_1\mu_j) \longrightarrow \T^n\times B(I_j,3\rho_j\mu_j)\] 
and 
\begin{eqnarray*}
H\circ\Psi_j & = & \mu_j \tilde{H}\circ\tilde{\Phi}_j \circ \sigma_{j}^{-1} \\
& = & \mu_j (l_j+\tilde{h}+\tilde{s}_j+\tilde{f}_j) \circ \sigma_{j}^{-1} \\
& = & \mu_j(l_j+\tilde{h}) \circ \sigma_{j}^{-1} + \mu_j\tilde{s}_j\circ \sigma_{j}^{-1} + \mu_j\tilde{f}_j\circ \sigma_{j}^{-1}.  
\end{eqnarray*}
Observe that $\mu_j(l_j+\tilde{h}) \circ \sigma_{j}^{-1}=h$, so we may set
\[ g_j=\mu_j\tilde{s}_j\circ \sigma_{j}^{-1}, \quad f_j=\mu_j\tilde{f}_j\circ \sigma_{j}^{-1}, \]
and write
\[ H\circ\Psi_j=h+g_j+f_j. \]
It is clear that $\{g_j,l_i\}=0$, for $i\in\{1,\dots,j\}$, and as $\partial_{\theta} f_j=\mu_j \partial_{\theta} \tilde{f}_j$, then
\[ |\partial_{\theta} f_j|_{C^{0}(\T^n \times B(I_j,2\rho_1\mu_j))} \MP e^{-m^{1/\alpha}}\mu_j\mu_1 <e^{-m^{1/\alpha}}\mu_j.\]
Finally, since
\[ \tilde{\Phi}_j=\tilde{\Phi}^j \circ \cdots \circ \tilde{\Phi}^1 : \mathcal{D}_{2\rho_1} \rightarrow \mathcal{D}_{3\rho_j} \]
with $|\tilde{\Phi}^i-\mathrm{Id}|_{\alpha,\tilde{L}_{j-i+1}} \MP T_i\mu_i$ for $i\in\{1,\dots,j\}$, then
\[ |\Pi_I\tilde{\Phi}_j-\mathrm{Id}_I|_{C^{0}(\mathcal{D}_{2\rho_1})}\MP \max_{i=1,\dots,j}\{T_i\mu_i\} \]
hence
\[ |\Pi_I\Psi_j-\mathrm{Id}_I|_{C^{0}(\T^n \times B(I_j,2\rho_1\mu_j))} \MP \mu_j\max_{i=1,\dots,n}\{T_i\mu_i\} \PM \mu_j, \]
where the last estimate follows from the fact that $T_i\mu_i\MP 1$ for $i\in\{1,\dots,j\}$ with a suitable implicit constant. This ends the proof.
\end{proof}

\subsection{The $C^k$-case}\label{ana2}

Now let us explain how one can obtain similar normal forms for finitely differentiable Hamiltonians, with of course only a polynomial bound on the remainder. Here also, as in \cite{Bou10}, we shall first construct a global normal form for perturbations of linear integrable Hamiltonians (Lemma~\ref{lemmelincomp2}) and then recover local normal forms for perturbations of general Hamiltonians (Proposition~\ref{propcomp2}).

\paraga We use the same notations as in the previous section, that is $\omega_1\in\R^n\setminus\{0\}$ is a $T_1$-periodic vector, $l_1(I)=\omega_1.I$ and we have parameters $m\in\N^*$ and $\mu_1>0$, while $\rho_1$ is already fixed. Given $k^*\in\N^*$, recall that we are assuming $k\geq k^*n+1$. The following result is due to the author (\cite{Bou10}).

\begin{lemma}\label{lemmelin2}
Consider the Hamiltonian $H=l_1+f$ defined on $\mathcal{D}_{3\rho_1}$, with $f\in C^k(\mathcal{D}_{3\rho_1})$ and $|f|_k<\mu_1$. Assume that
\begin{equation*}
T_1\mu_1 \PM 1, \quad mT_1\mu_1 \PM 1.
\end{equation*}
Then there exists a $C^{k-k^*}$ symplectic transformation
\[ \Phi^1 : \mathcal{D}_{2\rho_1} \rightarrow \mathcal{D}_{3\rho_1} \]
with $|\Phi^1-\mathrm{Id}|_{k-k^*} \MP T_1\mu_1$ such that
\[ H^1=H\circ\Phi^1=l_1+g^1+f^1 \]
with $\{g^1,l_1\}=0$ and the estimates
\[ |g^1|_{k-k^*} \MP \mu_1, \quad |f^1|_{k-k^*} \MP m^{-k^*}\mu_1,  \]
hold true.
\end{lemma}

The statement above is exactly Proposition 3.2 in \cite{Bou10}, where it is stated for $mT_1\mu_1\PE 1$ and $k^*=k-2$, but of course it also holds for smaller $m$, that is $mT_1\mu_1 \PM 1$, and for any integer $k^*\in\N^*$ smaller than $k$. 

\paraga Now as in the previous section, performing suitable ``compositions" of the above lemma, this readily gives us the following statement.

\begin{lemma}\label{lemmelincomp2}
Let $j\in\{1,\dots,n\}$, and consider the Hamiltonian $H=l_j+f$ defined on $\mathcal{D}_{3\rho_j}$, with $f\in C^k(\mathcal{D}_{3\rho_j})$ and $|f|_{k}\MP\mu_j$. Assume that $(A_i)$ is satisfied for $i\in\{1,\dots,j\}$. Then there exists a symplectic transformation
\[ \Phi_j=\Phi^j \circ \cdots \circ \Phi^1 : \mathcal{D}_{2\rho_1} \rightarrow \mathcal{D}_{3\rho_j} \]
where $\Phi^i : \mathcal{D}_{2\rho_i} \rightarrow \mathcal{D}_{3\rho_i}$ is $C^{k-k^*i}$ and $|\Phi^i-\mathrm{Id}|_{k-k^*i} \MP T_i\mu_i$ for $i\in\{1,\dots,j\}$, such that
\[ H_j=H\circ\Phi_j=l_j+g_j+f_j \in C^{k-k^*i}(\mathcal{D}_{2\rho_1}) \]
with $\{g_j,l_i\}=0$, for $i\in\{1,\dots,j\}$, and the estimates
\[ |g_j|_{k-k^*j} \MP \mu_1, \quad |f_j|_{k-k^*j} \MP m^{-k^*}\mu_1,  \]
hold true.
\end{lemma}

The proof is completely identical to the proof of Lemma~\ref{lemmelincomp2} (using Lemma~\ref{estimcomp2} instead of Lemma~\ref{estimcomp}), hence we do not repeat the details. 

\paraga Finally we come back to the original setting, that is 
\begin{equation*}
\begin{cases} 
H(\theta,I)=h(I)+f(\theta,I), \quad (\theta,I)\in \mathcal{D}_R, \\
|h|_{k}=1, \; |f|_{k} < \varepsilon, 
\end{cases}
\end{equation*}
for which we have the following proposition.

\begin{proposition}\label{propcomp2}
Suppose $H$ is as in~(\ref{HamC}), and assume that $(B_i)$ is satisfied for $i\in\{1,\dots,j\}$. Then there exists a $C^{k-k^*j}$ symplectic transformation 
\[\Psi_j :\T^n \times B(I_j,2\rho_1\mu_j)\rightarrow \T^n \times B(I_j,3\rho_j\mu_j)\] 
with $|\Pi_I\Psi_j-\mathrm{Id}_I|_{C^{0}(B(I_j,2\rho_1\mu_j))}\PM\mu_j$ such that 
\begin{equation*}
H \circ \Psi_j=h+g_j+f_j,
\end{equation*}
with $\{g_j,l_i\}=0$ for $i\in\{1,\dots,j\}$ and the estimate
\begin{equation*}
|\partial_{\theta} f_j|_{C^{0}(\T^n \times B(I_j,2\rho_1\mu_j))}<m^{-k^*}\mu_j
\end{equation*} 
holds true. 
\end{proposition} 

Once again, the proof is completely analogous to the proof of Proposition~\ref{propcomp} (using Lemma~\ref{cauchy2} instead of Lemma~\ref{cauchy}), hence there is no need to give further details.

Let us notice that our last estimate gives in fact
\[ |f_j|_{k-k^*j}<m^{-k^*}\mu_j, \]
but since $k-k^*j\geq 1$ for any $j\in\{1,\dots,n\}$, this yields in particular the estimate stated in the above proposition.

Here the case $j=1$ is due to author (\cite{Bou10}) and is enough to prove effective stability for quasi-convex unperturbed systems, but as we already explained, we shall need this result for any $j\in\{1,\dots,n\}$.

\subsection{Dynamical consequences}\label{ana3}

Let us now examine the dynamical consequences of our normal forms. As usual, it will be used to control the directions, if any, in which the action variables in these new coordinates can actually drift. Below we shall state a result in the coordinates given by our normal forms, and we shall come back to our original coordinates at the beginning of the next section.  

In order to treat the Gevrey case and the $C^k$ case in a unified way, we introduce yet another parameter $\tau_m>0$ which will eventually gives us the time of stability. For $\alpha$-Gevrey Hamiltonians, we set 
\[\tau_m=e^{m^{1/\alpha}},\] 
and for $C^k$-Hamiltonians, with $k\geq k^*n+1$ for some $k^*\in\N^*$, this will be 
\[\tau_m=m^{k^*}.\] 
Under the assumptions of Proposition~\ref{propcomp} or Proposition~\ref{propcomp2}, consider the Hamiltonian
\[ H_j=H\circ\Psi_j=h+g_j+f_j \]  
defined on the domain $\T^n \times B(I_j,4\mu_j)$ (from now on we shall use the fact that we have defined $\rho_1=2$). Let $\mathcal{M}_j$ be the $\Z$-module
\[ \mathcal{M}_j=\{k\in\Z^n \; | \; k.\omega_i=0, \; i\in\{1,\dots,j\}\}. \]
Since the periodic vectors are independent, its rank is $n-j$, and it is also the dimension of the vector space $\Lambda_j=\mathcal{M}_j \otimes \R$ spanned by $\mathcal{M}_j$. 

We shall need the following lemma, which is completely obvious using the definition of the Poisson bracket.

\begin{lemma} \label{resonant}
The equality $\{g_j,l_i\}=0$, for all $i\in\{1,\dots,j\}$, is equivalent to $\partial_\theta g_j(\theta,I) \in \Lambda_j$, for $(\theta,I)\in\T^n \times B(I_j,4\mu_j)$.
\end{lemma}
 
Now consider a solution $(\theta^j(t),I^j(t))$ of the Hamiltonian $H_j$ starting at $I^j(t_j) \in B(I_j,4\mu_j)$ for some $t_j \in \R$, and define the time of escape of this solution as the smallest time $\tilde{t}_j \in ]t_j,+\infty]$ for which $I^j(\tilde{t}_j) \notin B(I_j,4\mu_j)$. 

The only information we shall use from our normal forms is contained in the next proposition, where we shall denote by $\Pi_j$ the projection onto the linear subspace $\Lambda_j$.
  
\begin{proposition} \label{cor}
Let $H_j=h+g_j+f_j$ be a Hamiltonian defined on the domain $\T^n \times B(I_j,4\mu_j)$, with $\{g_j,l_i\}=0$ for $i\in\{1,\dots,j\}$, and such that the estimate
\begin{equation*}
|\partial_{\theta} f_j|_{C^{0}(\T^n \times B(I_j,4\mu_j))}<\tau_{m}^{-1}\mu_j
\end{equation*}  
hold true. Then, with the previous notations, we have
\[ |I^j(t)-I^j(t_j)-\Pi_j(I^j(t)-I^j(t_j))|<\mu_j, \quad t \in [t_j,\tau_m] \cap [t_j,\tilde{t}_j[. \] 
In particular, 
\[ |I^n(t)-I^n(t_n)|<\mu_n, \quad t \in [t_n,\tau_m]. \]
\end{proposition}

\begin{proof}
Let $\Pi_{j}^{\perp}$ be the projection onto the orthogonal complement of $\Lambda_j$, so that $\Pi_j+\Pi_{j}^{\perp}$ is the identity and therefore 
\[ |I^j(t)-I^j(t_j)-\Pi_j(I^j(t)-I^j(t_j))|=|\Pi_{j}^{\perp}(I^j(t)-I^j(t_j))|. \]
Now, as long as $t<\tilde{t}_j$, the equations of motion for $H_j=h+g_j+f_j$ and the mean value theorem give
\[ |I^j(t)-I^j(t_j)| \leq |t-t_j||\partial_\theta (g_j+f_j)|_{C^{0}(\T^n \times B(I_j,4\mu_j))}.  \]
But $\{g_j,l_i\}=0$ for $i\in\{1,\dots,j\}$, so by Lemma~\ref{resonant} we have $\partial_\theta g_j(\theta,I) \in \Lambda_j$ for any $(\theta,I)\in\T^n \times B(I_j,4\mu_j)$, hence if we first project the equations onto the orthogonal complement of $\Lambda_j$ we have
\[ |\Pi_{j}^{\perp}(I^j(t)-I^j(t_j))| \leq |t-t_j||\partial_\theta f_j|_{C^{0}(\T^n \times B(I_j,4\mu_j))}.  \]
Now since $|t-t_j|\leq \tau_m$ and $|\partial_\theta f_j|_{C^{0}(\T^n \times B(I_j,2\rho_1\mu_j))}<\tau_{m}^{-1}\mu_j$, the previous estimate gives 
\[ |\Pi_{j}^{\perp}(I^j(t)-I^j(t_j))|<\mu_j, \quad t \in [t_j,\tau_m[ \cap [t_j,\tilde{t}_j[, \]
and therefore
\[ |I^j(t)-I^j(t_j)-\Pi_j(I^j(t)-I^j(t_j))|<\mu_j, \quad t \in [t_j,\tau_m[ \cap [t_j,\tilde{t}_j[. \]  

Finally, for $j=n$, note that $g_n$ is integrable (since $\Lambda_n$) and $\Pi_{n}$ is identically zero, so that the mean value theorem immediately gives $\tilde{t}_n > \tau_m$ and the estimate
\[ |I^n(t)-I^n(t_n)|<\mu_n, \quad t \in [t_n,\tau_m], \]
follows easily. This concludes the proof.
\end{proof}

The interpretation of the above proposition is the following: if $\lambda_j$ is the affine subspace passing through $I^j(t_j)$ with direction space $\Lambda_j$, then as long as $I^j(t)$ remains in the domain of definition, it is $\mu_j$-close to $\lambda_j$ during an interval of time of length $\tau_m$. This means that for that interval of time, there is almost no variation of the action components in the direction transversal to $\lambda_j$, so that any potential drift has to occur along that space.  

If we were assuming that the energy sub-levels of the integrable Hamiltonian are convex, then some direct arguments using the preservation of energy would give us a complete stability result for such solutions. But in our more general situation, we will have to use indirect and more complicated geometric arguments.

\section{Geometric part}\label{geo}

In this section, we shall describe some geometric arguments, first introduced by Niederman (\cite{Nie04}, see also \cite{BN09} for a somehow clearer exposition), that will lead to the proof of both Theorem~\ref{thmG} and Theorem~\ref{thmC}. Without loss of generality, we will consider only solutions $(\theta(t),I(t))$ starting at time $t_0=0$ and evolving in positive time $t>0$. In section~\ref{geo1}, we will introduce a class of solutions, which we call ``restrained", and for which the stability of the action variables is easily proved. Then, in section~\ref{geo2}, we introduce the notion of ``drifting" solutions, which by definition do not satisfy the stability properties implied by restrained solutions. We will then show that drifting solutions can in fact be restrained provided some assumptions are required, hence leading to the non-existence of such drifting solutions. This will eventually give us a proof of Theorem~\ref{thmG} and Theorem~\ref{thmC} in section~\ref{geo3}. 

As these geometric arguments are not affected at all by the regularity of the system, they will be similar for Gevrey or $C^k$ Hamiltonians. Of course, they are also the same for analytic Hamiltonians which were studied in \cite{BN09}, therefore we shall merely state and explain the relevant results, for which detailed proofs are available in \cite{BN09}.

\subsection{Restrained solutions}\label{geo1}

\paraga In order to define our restrained solutions, we shall need some notations. Recall from Proposition~\ref{propcomp} that if the suitable assumptions are met, we can define transformations
\[\Psi_j :\T^n \times B(I_j,4\mu_j)\rightarrow \T^n \times B(I_j,3\rho_j\mu_j)\] 
with $|\Pi_I\Psi_j-\mathrm{Id}_I|_{C^{0}(B(I_j,4\mu_j))}\PM\mu_j$, for $j\in\{1,\dots,n\}$. In particular, we can easily assume that the the image of $\Psi_j$ contains the domain $\T^n \times B(I_j,2\mu_j)$. From now on, we shall write
\[ \mathcal{B}_j=B(I_j,2\mu_j), \quad j\in\{1,\dots,n\}, \]
and for completeness we set $\mathcal{B}_0=B_R$.

Now consider a solution $(\theta(t),I(t))\in\T^n\times\mathcal{B}_0$ of our Hamiltonian $H$, starting at time $t_0=0$. If at some time $t_1\geq 0$ we can find a periodic vector $\omega_1\in\R^n\setminus\{0\}$ such that 
\[ |\nabla h(I(t_1))-\omega_1|<\mu_1,\] 
then setting $I_1=I(t_1)$, $(\theta(t_1),I(t_1))\in\T^n\times\mathcal{B}_1$. Therefore, if assumption $(B_1)$ is satisfied, we can define a normalized solution $(\theta^1(t),I^1(t))$, for $t\geq t_1$, by
\[ \Psi_1(\theta^1(t),I^1(t))=(\theta(t),I(t)), \]
as long as $I(t)\in \mathcal{B}_1$. Now we can start again, but this time with the solution $(\theta^1(t),I^1(t))\in\T^n\times\mathcal{B}_1$ of the Hamiltonian $H_1=H\circ\Psi_1$ starting at time $t_1$: if at some time $t_2\geq t_1$ we can find a periodic vector $\omega_2\in\R^n\setminus\{0\}$, independent of $\omega_1$, such that 
\[ |\nabla h(I^1(t_2))-\omega_2|<\mu_2,\] 
then setting $I_2=I^2(t_1)$, $(\theta^1(t_2),I^1(t_2))\in\T^n\times\mathcal{B}_2$, and provided $(B_2)$ holds, we can define yet another normalized solution $(\theta^2(t),I^2(t))$, for $t \geq t_2$, by
\[ \Psi_2(\theta^2(t),I^2(t))=(\theta^1(t),I^1(t)), \]
as long as $I^1(t)\in \mathcal{B}_2$. 

Inductively, setting $(\theta^0(t),I^0(t))=(\theta(t),I(t))$, for $j\in\{1,\dots,n\}$ we can define the averaged solution $(\theta^j(t),I^j(t))$, for $t\geq t_j$, by
\[ \Psi_j(\theta^j(t),I^j(t))=(\theta^{j-1}(t),I^{j-1}(t)), \]
as long as $I^{j-1}(t)\in \mathcal{B}_j$, provided we have found independent periodic vectors $\omega_1,\dots, \omega_j$ such that
\[ |\nabla h(I_j)-\omega_j|<\mu_j,\] 
with $I_j=I^{j-1}(t_j)$ and assuming $(B_j)$ is satisfied. 

Moreover, using our estimate on $\Psi_j$ we have
\begin{equation*}
|I^j(t)-I^{j-1}(t)| \PM \mu_j,  \quad j\in\{1,\dots,n\},
\end{equation*}
during that time interval. 

\paraga  We can eventually write our definition.

\begin{definition}
Given $\mu_0>0$ and $m\in\N^*$, a solution $(\theta(t),I(t))$ of the Hamiltonian~(\ref{HamG}) or~(\ref{HamC}), starting at time $t_0=0$, is said to be {\it restrained} (by $\mu_0$, up to time $\tau_m$) if we can find sequences of: 
\begin{itemize}
\item[(1)] radii $(\mu_1,\dots,\mu_n)$, with $0<\mu_n<\cdots<\mu_1<\mu_0$; 
\item[(2)] independent periodic vectors $(\omega_1,\dots,\omega_n)$, with periods $(T_1,\dots,T_n)$; 
\item[(3)] times $(t_1,\dots,t_n)$, with $0=t_0 \leq t_1 \leq \cdots \leq t_n \leq t_{n+1}=\tau_m$, 
\end{itemize}
satisfying, for $j\in\{0,\dots,n-1\}$, conditions $(B_{j+1})$ and the following conditions~$(C_j)$ defined by
\begin{equation} \label{Cj}
\begin{cases}
|I^j(t)-I^j(t_j)| < \mu_j, \quad \ t \in [t_j,t_{j+1}], \\ 
|\nabla h(I^j(t_{j+1}))-\omega_{j+1}|<\mu_{j+1}.
\end{cases} \tag{$C_j$}
\end{equation}
\end{definition}

Let us notice that for $j\in\{0,\dots,n-2\}$, setting $I_{j+1}=I^j(t_{j+1})$ the second condition of $(C_j)$ gives part of condition $(B_{j+1})$, and also it ensures that the first condition of $(C_{j+1})$ is indeed well-defined. Also, our assumptions imply the inclusion of domains $\mathcal{B}_{j+1} \subseteq \mathcal{B}_j$ for $j\in\{1,\dots,n-1\}$ (which was the assumption made in \cite{BN09}): indeed, one can choose an implicit constant in $(B_{j+1})$ so that $2\mu_{j+1}\leq \mu_j$, so if $I\in \mathcal{B}_{j+1}$, then
\begin{eqnarray*}
|I-I_j| & \leq & |I-I_{j+1}|+|I_{j+1}-I_j| \\
& \leq & |I-I_{j+1}|+|I^j(t_{j+1})-I^{j-1}(t_{j})|\\
& \leq & 2\mu_{j+1}+\mu_j \leq 2\mu_j.
\end{eqnarray*}

\paraga The terminology ``restrained" was introduced in \cite{BN09} because for such solutions, the actions $I(t)$ (or some properly normalized actions $I^{j}(t)$) are forced to pass close to a resonance at the time $t=t_j$, the multiplicity of which decreases as $j$ increases (since the periodic vectors are assumed to be independent), and moreover the variation of these (normalized) actions is controlled on each time interval $[t_j,t_{j+1}]$. Hence after the time $t_n$, the actions are in a domain free of resonances and are easily confined in view of the last part of Proposition~\ref{cor}. This is the content of the above proposition. 

\begin{proposition} \label{etape1}
Consider a restrained solution $(\theta(t),I(t))$, with an initial action $I(0)\in B_{R/2}$. If $\mu_0\PM 1$, then the estimates
\[ |I(t)-I(0)| < (n+1)^2\mu_0, \quad 0\leq t\leq \tau_m, \]
hold true.
\end{proposition}

This is exactly Proposition 3.7 in \cite{BN09}, to which we refer for the easy proof.

\subsection{Drifting solutions}\label{geo2}

Restrained solutions are stable for an exponentially long interval of time with respect to $m$, and now we will show that this is in fact true for all solutions. 

\paraga The following definition will be useful in the sequel.

\begin{definition}
Given $\mu_0>0$ and $m\in\N^*$, a solution $(\theta(t),I(t))$ of the Hamiltonian~(\ref{HamG}) or~(\ref{HamC}), starting at time $t_0=0$, is said to be drifting (to $\mu_0$, before time $\tau_m$) if there exists a time $t_*$ satisfying  
\[ |I(t_*)-I(0)|=(n+1)^2\mu_0, \quad 0<t_*\leq \tau_m. \]
\end{definition} 

Of course, this definition makes sense only if $(n+1)^2\mu_0<R/2$. In view of Proposition~\ref{etape1}, drifting solutions cannot be restrained. However, we will prove below that if such a drifting solution exists, it has to be restrained under some assumptions on $\mu_0$, $m$ and $\varepsilon$, which will eventually prove that drifting solutions do not exist so that all solutions are indeed exponentially stable. 

More precisely, assuming the existence of a drifting solution, we will construct a sequence of radii $(\mu_1,\dots,\mu_n)$, an increasing sequence of times $(t_1,\dots,t_n)$ and a sequence of linearly independent vectors $(\omega_1,\dots,\omega_n)$, with periods $(T_1,\dots,T_n)$ satisfying, for $j\in\{0,\dots,n-1\}$, assumptions $(B_{j+1})$ and $(C_j)$. All sequences will be built inductively, and we first describe two lemmas that we shall need. 

\paraga For $j\in\{1,\dots,n\}$, recall that $\Lambda_j$  is the vector space spanned by 
\[ \mathcal{M}_j= \{k\in\Z^n \; | \; k.\omega_i=0, \; i\in\{1,\dots,j\}\}, \] 
and that $\Pi_j$ (resp. $\Pi_{j}^{\perp}$) is the projection onto $\Lambda_j$ (resp. $\Lambda_{j}^{\perp}$). Let us define the integer 
\[ L_j=\sup_{i\in\{1,\dots,j\}}\{|T_i\omega_i|\}\in\N^*, \quad j\in\{1,\dots,n-1\}. \]
For completeness, we set $\Lambda_0=\R^n$, $L_0=1$ and in this case $\Pi_0$ is nothing but the identity. 

The first lemma will allow us to construct the sequence of times, and for that we will rely on the fact that our integrable part $h$ belongs to $DM_{\gamma}^{\tau}(B)$, so that it satisfies the following steepness property (see Lemma 3.9 in \cite{BN09}).

\begin{lemma}\label{l1}
For $j\in\{0,\dots,n-1\}$, let $\lambda_j$ be any affine subspace with direction $\Lambda_j$, and take $c_j<1$. Then for any continuous curve $\Gamma_j : [t_j,t_j^*] \rightarrow \lambda_j \cap B_R$ with length 
\[ \vert\Gamma_j(t_j^*)-\Gamma_j(t_j)\vert =c_j \PM \gamma L_j^{-\tau},\] 
there exists a time $t_{j+1}\in [t_j,t_j^*]$ such that 
\begin{equation*}
\begin{cases}
\vert\Gamma_j(t)-\Gamma_j (t_j)\vert < c_j, \quad t\in [t_j,t_{j+1}], \\
\left\vert\Pi_j(\nabla h(\Gamma_j (t_{j+1})))\right\vert \PS c_j^2. 
\end{cases}
\end{equation*}
\end{lemma} 

Now to construct the sequence of periodic vectors, we shall use the following lemma, which is a straightforward application of Dirichlet's theorem on simultaneous Diophantine approximation (see Lemma 3.10 in \cite{BN09}).

\begin{lemma}\label{l2}
Given any vector $v \in \R^n$ and any real number $Q>0$, there exists a $T$-periodic vector $\omega$ satisfying
\[ |v-\omega| \leq T^{-1}Q^{-\frac{1}{n-1}}, \quad |v|^{-1} \leq T \leq Q|v|^{-1}. \]
\end{lemma} 

\paraga Now we have the necessary tools to show that drifting solutions cannot exist, provided suitable assumptions on the parameters are demanded. This will be done inductively, and for technical reasons we separate the first step (Proposition~\ref{etape2}) from the general inductive step (Proposition~\ref{etape3}).

\begin{proposition}\label{etape2}
Let $(\theta(t),I(t))$ be a drifting solution. If $\mu_0 \PM \gamma$, then there exist a time $t_1$, a $T_1$-periodic vector $\omega_1$ and $\mu_1 \EP T_1^{-1}\varepsilon^{a_1}$ for some positive constant $a_1$, satisfying 
\begin{itemize}
\item[$(a)$] $|I(t)-I(0)| < \mu_0,\quad t \in [0,t_1]$; 
\item[$(b)$] $|\nabla h(I(t_1))-\omega_1| < \mu_1$.
\end{itemize}
Moreover, we have the estimate
\begin{equation*}
1 \MP T_1 \MP \varepsilon^{-a_{1}(n-1)}r_{0}^{-2}, \quad 1 \leq L_1 \MP \varepsilon^{-a_{1}(n-1)}\mu_{0}^{-2}.
\end{equation*} 
\end{proposition}

The proof is completely analogous to the proof of Proposition 3.11 in \cite{BN09}: it uses the fact that our solution is drifting, Lemma~\ref{l1} applied to the curve $\Gamma_0(t)=I(t)$ with $c_0=\mu_0$, and Lemma~\ref{l2}. 

\begin{proposition} \label{etape3}
Let $(\theta(t),I(t))$ be a drifting solution, $j \in \{1,\dots,n-1\}$ and assume that there exist sequences $(t_1,\dots,t_j)$, $(\omega_1,\dots,\omega_j)$ linearly independent and $(\mu_1, \dots, \mu_j)$, satisfying assumptions $(B_{i})$ and $(C_{i-1})$, for $i \in \{1,\dots,j\}$. Assume also that
\begin{itemize}
\item[$(i)$] $\left(T_j\mu_jL_j^{-1}\right)^\tau\PM\mu_j$;
\item[$(ii)$] $\left(T_j\mu_jL_j^{-1}\right)^\tau \PM \gamma L_j^{-\tau}$;
\item[$(iii)$] $\mu_1 \PM \mu_0^2$.
\end{itemize}
Then there exist a time $t_{j+1}$, a $T_{j+1}$-periodic vector $\omega_{j+1}$ and $\mu_{j+1} \EP T_{j+1}^{-1}\varepsilon^{a_{j+1}}$ for some positive constant $a_{j+1}$, satisfying 
\begin{itemize}
\item[$(a)$] $|I^j(t)-I^j(t_j)| < \mu_j,\quad t \in [t_j,t_{j+1}]$;
\item[$(b)$] $|\nabla h(I^j(t_{j+1}))-\omega_{j+1}| < \mu_{j+1}$;
\item[$(c)$] $|\omega_{j+1}-\omega_j|\MP \mu_j$.
\end{itemize}
Moreover, we have the estimates
\begin{equation*}
1 \MP T_{j+1} \MP \varepsilon^{-a_{j+1}(n-1)}\mu_{0}^{-2}, \quad 1 \leq L_{j+1} \MP \max_{i\in\{1,\dots,j+1\}}\{\varepsilon^{-a_i(n-1)}\}\mu_0^{-2},
\end{equation*}  
and if
\begin{itemize}
\item[$(iv)$] $\mu_{j+1}\PM\left(T_j\mu_jL_j^{-1}\right)^{2\tau}$,
\end{itemize}
then $\omega_{j+1}$ is linearly independent of $(\omega_1,\dots,\omega_j)$.
\end{proposition}

Once again, the proof is completely similar to the proof of Proposition 3.12 in \cite{BN09}. As before, it uses the fact that our solution is drifting, Lemma~\ref{l1} applied to the curve $\Gamma_j(t)=I^j(t_j)+\Pi_j(I^j(t)-I^j(t_j))$ with $c_j=\left(T_j\mu_jL_j^{-1}\right)^\tau$, Lemma~\ref{l2} and Proposition~\ref{cor}. 

\subsection{Proof of Theorem~\ref{thmG} and Theorem~\ref{thmC}}\label{geo3}

We can finally prove our theorems, and once again we refer to the proof of Theorem 2.4 in \cite{BN09} for some more details.

\begin{proof}[Proof of Theorem~\ref{thmG} and Theorem~\ref{thmC}]
As a consequence of Propositions~\ref{etape1}, \ref{etape2} and a repeated use of Proposition~\ref{etape3}, we know that
\[ |I(t)-I(0)| < (n+1)^2\mu_0, \quad 0\leq t \leq \tau_m \]
provided that the parameters $\mu_0$, $m$ and $\varepsilon$ satisfy the following eleven conditions:
\begin{itemize}
\item[$(i)$] $\mu_{j+1} \PM \left(T_j\mu_jL_j^{-1}\right)^{2\tau}$, $j \in \{1,\dots,n-1\}$;
\item[$(ii)$] $\left(T_j\mu_jL_j^{-1}\right)^\tau \PM \mu_j$, for $j \in \{1,\dots,n-1\}$;
\item[$(iii)$] $mT_j\mu_j \PM 1$, for $j \in \{1,\dots,n\}$;
\item[$(iv)$] $\mu_1 \PM \mu_0^2$;
\item[$(v)$] $\varepsilon< \mu_j^2$, for $j \in \{1,\dots,n\}$;
\item[$(vi)$] $\left(T_j\mu_jL_j^{-1}\right)^\tau\PM\gamma L_j^{-\tau}$, for $j \in \{1,\dots,n-1\}$;
\item[$(vii)$] $T_j\mu_j \PM 1$, for $j \in \{1,\dots,n\}$; 
\item[$(viii)$] $\mu_j \PM 1$, for $j \in \{1,\dots,n\}$;
\item[$(ix)$] $\mu_{j}\PM \mu_{j-1}$, for $j \in \{2,\dots,n\}$;
\item[$(x)$] $\mu_0\PM \gamma$;
\item[$(xi)$] $\mu_0 \PM 1$;
\end{itemize}
where $\mu_j \EP T_{j}^{-1}\varepsilon^{a_j}$, with $a_j$ to be chosen for $j\in\{1,\dots,n\}$, and 
\begin{equation*}
1 \MP T_j \MP \varepsilon^{-a_{j}(n-1)}\mu_{0}^{-2}, \quad 1 \leq L_j \MP \max_{i\in\{1,\dots,j\}}\{\varepsilon^{-a_i(n-1)}\}\mu_0^{-2}.  
\end{equation*}
So let us choose $m \PE \varepsilon^{-a}$ and $r_0=\varepsilon^b$, for two positive constants $a$ and $b$. One can check (this is done in details in \cite{BN09}) that all these conditions hold true if we choose
\[ a_j=(2\tau (n+1))^{-n-1+j}, \quad j\in\{1,\dots,n\} \]
and
\[ a=b=3^{-1}(2\tau (n+1))^{-n}, \] 
provided $\varepsilon \leq \varepsilon_0$, with a sufficiently small $\varepsilon_0$ depending on $n,R,\alpha,L,E,M,\gamma$ and $\tau$ (resp. on $n,R,k,M,\gamma$ and $\tau$) in the Gevrey case (resp. in the $C^k$ case). Now recalling that for $\alpha$-Gevrey Hamiltonians, 
\[\tau_m=e^{m^{1/\alpha}},\] 
and for $C^k$-Hamiltonians, with $k\geq k^*n+1$ for $k^*\in\N^*$,
\[\tau_m=m^{k^*},\] 
this completes the proof of both theorems.
\end{proof}

\appendix

\section{Technical estimates}\label{tech}

In this short appendix, we give some technical estimates concerning Gevrey and finitely differentiable functions that we used in section~\ref{ana}.

\paraga First in the proof of Lemma~\ref{lemmelincomp}, we used the following estimate concerning the composition of Gevrey functions.

\begin{lemma}\label{estimcomp}
Let $0<\rho'<\rho$ and $L'=CL$. Suppose that $g\in G^{\alpha,L}(D_\rho)$ and that $\Phi : \mathcal{D}_{\rho'} \rightarrow \mathcal{D}_\rho$ is $(\alpha,L')$-Gevrey. If
\[ |\Phi-\mathrm{Id}|_{\alpha,L'}\MP 1, \]
then $f\circ\Phi \in G^{\alpha,L'}(\mathcal{D}_{\rho'})$ and
\[ |f\circ\Phi|_{\alpha,L'} \leq |f|_{\alpha,L}. \] 
\end{lemma}

This lemma is contained in the statement of Corollary A.1, Appendix A.2, in \cite{MS02}, to which we refer for a proof and a possible choice of implicit constant.

In fact, we could have used a more elaborated statement concerning compositions of Gevrey vector-valued functions (as in Proposition A.1 in \cite{MS02}), which would have implied that the diffeomorphisms $\Phi_j$ constructed in Lemma~\ref{lemmelincomp} are in fact $(\alpha,L_j)$-Gevrey, with an estimate on their distance to the identity. However, such a statement is not very elegant to state and it is not needed here.

\paraga Then, in the proof of Proposition~\ref{propcomp}, we used the following lemma which enabled us to bound the Gevrey norm of the derivatives of a function in terms of the Gevrey norm of the function.

\begin{lemma}\label{cauchy}
Let $\rho>0$ and $g\in G^{\alpha,L}(\mathcal{D}_\rho)$. For $p\in\N$, we have
\[ \sum_{l\in\N^{2n},\,|l|=p}|\partial^l g|_{\alpha,L/2}\MP|g|_{\alpha,L}. \]
\end{lemma}

For an easy proof and a possible implicit constant, we refer to Lemma A.1, Appendix A.1 in \cite{MS02}. 

\paraga Finally, we shall also need corresponding estimates for the $C^k$ norms, which are well-known and much more easy to prove.

The above lemma is the analogue of Lemma~\ref{estimcomp}, and it is needed in the proof of Lemma~\ref{lemmelincomp2}.

\begin{lemma}\label{estimcomp2}
Let $0<\rho'<\rho$, suppose that $g\in C^k(D_\rho)$ and that $\Phi : \mathcal{D}_{\rho'} \rightarrow \mathcal{D}_\rho$ is of class $C^k$. If
\[ |\Phi-\mathrm{Id}|_{k}\MP 1, \]
then $f\circ\Phi \in C^k(\mathcal{D}_{\rho'})$ and
\[ |f\circ\Phi|_{k} \leq |f|_{k}. \] 
\end{lemma}  

Finally, here's an easy analogue of Lemma~\ref{cauchy} which is useful in the proof of Proposition~\ref{propcomp2}. 

\begin{lemma}\label{cauchy2}
Let $\rho>0$ and $g\in C^k(\mathcal{D}_\rho)$. For $0\leq p\leq k$, we have
\[ \sum_{l\in\N^{2n},\,|l|=p}|\partial^l g|_{k-p}\MP|g|_{k}. \]
\end{lemma}

{\it Acknowledgments.} This work has been written up while the author served as a Research Fellow at Warwick University, through CODY network.

\addcontentsline{toc}{section}{References}
\bibliographystyle{amsalpha}
\bibliography{NekGC}

\end{document}